\documentclass[a4paper,11pt]{article}
\usepackage{amsmath, amsfonts, amscd, amssymb, amsthm, enumerate, xypic}

\def\Mat{\text{M}}
\def\GL{\text{GL}}

\def\card{\#\,}

\newcommand{\Vect}{\operatorname{span}}

\newcommand{\tr}{\operatorname{tr}}
\newcommand{\rk}{\operatorname{rk}}
\newcommand{\Sp}{\operatorname{Sp}}
\newcommand{\codim}{\operatorname{codim}}
\renewcommand{\setminus}{\smallsetminus}

\def\F{\mathbb{F}}
\def\K{\mathbb{K}}
\def\R{\mathbb{R}}

\def\N{\mathbb{N}}

\def\calH{\mathcal{H}}

\def\calV{\mathcal{V}}

\def\lcro{\mathopen{[\![}}
\def\rcro{\mathclose{]\!]}}

\theoremstyle{definition}

\theoremstyle{plain}
\newtheorem{theo}{Theorem}
\newtheorem{prop}[theo]{Proposition}
\newtheorem{cor}[theo]{Corollary}

\theoremstyle{plain}

\theoremstyle{remark}
\newtheorem{Rems}{Remarks}
\newtheorem{Rem}[Rems]{Remark}

\title{On the matrices of given rank in a large subspace}
\author{Cl\'ement de Seguins Pazzis\footnote{Professor of Mathematics at Lyc\'ee Priv\'e Sainte-Genevi\`eve, 2, rue
de l'\'Ecole des Postes, 78029 Versailles Cedex, FRANCE.}
\footnote{e-mail address: dsp.prof@gmail.com}}

\begin{document}

\thispagestyle{plain}
\maketitle

\begin{abstract}
Let $V$ be a linear subspace of $\Mat_{n,p}(\K)$ with codimension lesser than $n$, where $\K$ is an arbitrary field
and $n \geq p$.
In a recent work of the author, it was proven that $V$ is always spanned by its
rank $p$ matrices unless $n=p=2$ and $\K \simeq \F_2$. Here, we give a sufficient condition on $\codim V$
for $V$ to be spanned by its rank $r$ matrices for a given $r \in \lcro 1,p-1\rcro$.
This involves a generalization of the Gerstenhaber theorem on linear subspaces of nilpotent matrices. 
\end{abstract}



\section{Introduction}

In this paper, $\K$ denotes an arbitrary field, $n$ a positive integer and
$\Mat_n(\K)$ the algebra of square matrices of order $n$ with entries in $\K$.
For $(p,q)\in \N^2$, we also let $\Mat_{p,q}(\K)$ denote the vector space of matrices with $p$ rows, $q$
columns and entries in $\K$. Two linear subspaces $V$ and $W$ of $\Mat_{p,q}(\K)$ will be called equivalent when there are non-singular
matrices $P$ and $Q$ respectively in $\GL_p(\K)$ and $\GL_q(\K)$ such that $W=P\,V\,Q$.

\noindent In a recent work of the author \cite{dSPclass}, the following proposition was a major tool for generalizing
a theorem of Atkinson and Lloyd \cite{AtkLloyd} to an arbitrary field:

\begin{prop}\label{genrangmax}
Let $n$ and $p$ denote positive integers such that $n \geq p$.
Let $V$ be a linear subspace of $\Mat_{n,p}(\K)$ such that $\codim V<n$,
and assume $(n,p,\card \K) \neq (2,2,2)$ or $\codim V<n-1$. Then $V$ is spanned by its rank $p$ matrices.
\end{prop}

\noindent The exceptional case of $\Mat_2(\F_2)$ is easily described:

\begin{prop}\label{oddcase}
Let $V$ be a linear hyperplane of $\Mat_2(\F_2)$. Then:
\begin{itemize}
\item either $V$ is equivalent to $\frak{sl}_2(\F_2)=\bigl\{M \in M_2(\F_2) : \; \tr M=0\bigr\}$
and then $V$ is spanned by its rank $2$ matrices;
\item or $V$ is equivalent to the subspace $T_2^+(\F_2)$ of upper triangular matrices, and then
$V$ is not spanned by its rank $2$ matrices.
\end{itemize}
\end{prop}

\begin{proof}
Consider the orthogonal $V^\bot$ of $V$ for the non-degenerate symmetric bilinear form $b : (A,B) \mapsto \tr(AB)$.
Then $V^\bot$ contains only one non-zero matrix $C$. Either $C$ has rank $2$, and it is equivalent to $I_2$, hence $V$
is equivalent to $\frak{sl}_2(\F_2)$; or $C$ has rank $1$, it is equivalent to $\begin{bmatrix}
0 & 1 \\
0 & 0
\end{bmatrix}$ hence $V$ is equivalent to $T_2^+(\F_2)$. In the first case, the three non-singular matrices
$I_2$, $\begin{bmatrix}
1 & 1 \\
0 & 1
\end{bmatrix}$ and $\begin{bmatrix}
1 & 0 \\
1 & 1
\end{bmatrix}$ span $\frak{sl}_2(\F_2)$. In the second one, $T_2^+(\F_2)$ has only two non-singular matrices,
which obviously cannot span it.
\end{proof}

\noindent Here, we wish to give a similar result for the rank $r$ matrices, still assuming that
$\codim V<n$. Our main results follow:

\begin{theo}\label{LCinf}
Let $n \geq p$ be integers and $V$ be a linear subspace of $\Mat_{n,p}(\K)$ with $\codim V<n$. \\
Let $r \in \lcro 1,p\rcro$ and $s \in \lcro 0,r\rcro$. Then every rank $s$ matrix of $V$
is a linear combination of rank $r$ matrices of $V$, unless $n=p=r=\# \K=2$ and $\codim V=1$.
\end{theo}

\noindent This has the following easy corollary (which will be properly proven later on):

\begin{cor}\label{exist}
Let $n \geq p$ be integers and $V$ be a linear subspace of $\Mat_{n,p}(\K)$ with $\codim V<n$. \\
Then, for every $r \in \lcro 1,p\rcro$, the subspace $V$ contains a rank $r$ matrix.
\end{cor}

\begin{theo}\label{condsuff}
Let $n \geq p$ be integers and $V$ be a linear subspace of $\Mat_{n,p}(\K)$ with $\codim V<n$. \\
Let $r \in \lcro 1,p-1\rcro$. If $\codim V \leq \binom{r+2}{2}-2$, then $V$ is spanned by its rank $r$ matrices.
\end{theo}

\noindent Notice that this has the following nice corollary (for which a much more elementary proof exists):

\begin{cor}\label{corhyper}
Let $H$ be a linear hyperplane of $\Mat_n(\K)$, with $n \geq 2$. Then $H$ is spanned by its rank $r$
matrices, for every $r \in \lcro 1,n\rcro$, unless $(n,r,\# \K)=(2,2,2)$.
\end{cor}

We will also show that there exists a linear subspace of $\Mat_{n,p}(\K)$
with codimension $\binom{r+2}{2}-1$ which is not spanned by its rank $r$ matrices, hence the upper bound
$\binom{r+2}{2}-2$ in the above theorem is tight.
The proof of these results will involve an extension of the Flanders theorem to affine subspaces (see Section 3 of \cite{dSPaffpres})
and a slight generalization of the famous Gerstenhaber theorem \cite{Gerstenhaber} on linear subspaces of nilpotent matrices.

\section{Proving the main theorems}\label{proofs}

\subsection{Proof of Proposition \ref{genrangmax}}

For the sake of completeness, we will recall here the proof of Proposition \ref{genrangmax}, already featured in \cite{dSPclass}.
This is based on the following theorem of the author, slightly generalizing earlier works of Dieudonn\'e
\cite{Dieudonne}, Flanders \cite{Flanders} and Meshulam \cite{Meshulam}:

\begin{theo}\label{FlandersdSP}
Given positive integers $n \geq p$, let $\calV$ be an affine subspace of $\Mat_{n,p}(\K)$
containing no rank $p$ matrix.
Then $\codim \calV \geq n$. \\
If in addition $\codim \calV=n$ and $(n,p,\card \K) \neq (2,2,2)$, then
$\calV$ is a linear subspace of $\Mat_{n,p}(\K)$.
\end{theo}

\begin{proof}[Proof of Proposition \ref{genrangmax}]
Assume that $V$ is not spanned by its rank $p$ matrices.
Then there would be a linear hyperplane $H$ of $V$ containing every rank $p$ matrix of $V$.
Choosing $M_0 \in V \setminus H$, it would follow that the affine
subspace $M_0+H$, which has codimension in $\Mat_{n,p}(\K)$ lesser than or equal to $n$, contains no rank $p$ matrix.
However $M_0+H$ is not a linear subspace of $\Mat_{n,p}(\K)$, which contradicts the above theorem.
\end{proof}

\subsection{Proof of Theorem \ref{LCinf} and Corollary \ref{exist}}

We discard the case $n=p=r=\card \K=2$ and $\codim V=1$, which has already been studied in the proof of Proposition \ref{oddcase}.

Let us now prove Theorem \ref{LCinf}. Let $A$ be a rank $s$ matrix of $V$.
Replacing $V$ with an equivalent subspace, we lose no generality assuming that $A$ has the form
$$A=\begin{bmatrix}
A_1 & 0
\end{bmatrix} \quad \text{for some $A_1 \in \Mat_{n,r}(\K)$.}$$
Denote by $W$ the linear space consisting of those matrices $M \in \Mat_{n,r}(\K)$ such that
$\begin{bmatrix}
M & 0
\end{bmatrix} \in V$.
Then the rank theorem shows that $\codim_{\Mat_{n,r}(\K)} W \leq \codim_{\Mat_{n,p}(\K)} V <n$.
Notice that the situation $n=r=2$ may not arise, hence Proposition \ref{genrangmax}
shows that $W$ is spanned by its rank $r$ matrices. In particular, the matrix $A_1$ is a linear combination of rank $r$
matrices of $W$, hence $A$ is a linear combination of rank $r$ matrices of $V$. This proves Theorem \ref{LCinf}.
\vskip 2mm
\noindent Let us now turn to Corollary \ref{exist}.
Denote by $V'$ the linear subspace of $V$
consisting of its matrices with all columns zero
starting from the second one. Then $\dim V'\geq n-\codim V>0$, hence $V'\neq \{0\}$, which proves that
$V$ contains a rank $1$ matrix $M$. Then, for every $r \in \lcro 1,p\rcro$,
Theorem \ref{LCinf} shows that $M$ is a linear combination of rank $r$ matrices of $V$,
hence $V$ must contain at least one rank $r$ matrix!

\subsection{Proof of Theorem \ref{condsuff}}

We will start from an observation that is similar to the one that lead to Proposition \ref{genrangmax}.
Let $V$ be a linear subspace of $\Mat_{n,p}(\K)$, let $r \in \lcro 1,p-1\rcro$ and assume that
$V$ is not spanned by its rank $r$ matrices. Then there would be a linear hyperplane $H$ of $V$
containing every rank $r$ matrix of $V$. By Theorem \ref{LCinf},
the subspace $H$ must also contain every matrix of $V$ with rank lesser than or equal to $r$. Choosing arbitrarily $M_0 \in V \setminus H$,
it would follow that the (non-linear) affine subspace $M_0+H$ contains only matrices of rank greater than $r$ and has dimension $\dim V-1$.

Conversely, assume there exists an affine subspace $\calH$ of $\Mat_{n,p}(\K)$ which contains only matrices of rank greater than $r$
(notice then that $0 \not\in \calH$),
and let $H$ denote its translation vector space. Then $H$ must contain every rank $r$ matrix of the linear space $V':=\Vect \calH$,
therefore $V'$, which has dimension $\dim H+1$, is not spanned by its rank $r$ matrices.

Theorem \ref{condsuff} will thus come from the following result (applied to $k=r+1$),
which generalizes a theorem of Meshulam \cite{Meshulam2} to an arbitrary field and rectangular matrices
(Meshulam tackled the case of an algebraically closed field and the one of $\R$, and he restricted his study to square matrices).

\begin{theo}
Let $n \geq p \geq k$ be positive integers. Denote by
$h(n,p,k)$ the largest dimension for an affine subspace $\calV$ of $\Mat_{n,p}(\K)$
satisfying
$$\forall M \in \calV, \quad \rk M \geq k.$$
Then
$$h(n,p,k)=np-\binom{k+1}{2}\cdot$$
\end{theo}

Inequality $h(n,p,k) \geq np-\binom{k+1}{2}$ is obtained as in \cite{Meshulam2} by considering the
affine subspace
$\calH$ consisting of all $n \times p$ matrices of the form
$$\begin{bmatrix}
I_r +T & ? \\
? & ? \\
\end{bmatrix} \quad \text{with $T \in T_k^{++}(\K)$,}$$
where $T_k^{++}(\K)$ denotes the set of strictly upper triangular matrices of $\Mat_k(\K)$. \\
Obviously $\codim_{\Mat_{n,p}(\K)} \calH=\codim_{\Mat_k(\K)} T_k^{++}(\K)=\binom{k+1}{2}$,
whilst, judging from its left upper block, every matrix of $\calH$ has a rank greater than or equal to $k$.

\vskip 2mm
In order to prove that $h(n,p,k) \leq np-\binom{k+1}{2}$, we let $\calV$ be an arbitrary affine subspace of
$\Mat_{n,p}(\K)$ such that $\forall M \in \calV, \; \rk M \geq k$, and we prove that $\dim \calV \leq np-\binom{k+1}{2}\cdot$
Proceeding by downward induction on $k$, we may assume furthermore that $\calV$ contains a rank $k$ matrix.
We then lose no generality assuming that $\calV$ contains the matrix
$J_k:=\begin{bmatrix}
I_k & 0 \\
0 & 0
\end{bmatrix}$. Denote by $V$ the translation vector space of $\calV$ and consider the linear
subspace $W$ of $\Mat_k(\K)$ consisting of those matrices $A$ for which
$\begin{bmatrix}
A & 0 \\
0 & 0
\end{bmatrix}$ belongs to $V$. Then the rank theorem shows $\codim_{\Mat_{n,p}(\K)} V \geq \codim_{\Mat_k(\K)} W$.
The assumptions on $\calV$ show that $I_k+W$ contains only non-singular matrices.
Since $W$ is a linear subspace, this shows that for any $M \in W$, the only possible eigenvalue of $M$ in the field $\K$ is $0$.
The proof will thus be finished should we establish the next theorem:

\vskip 2mm
\noindent For $M \in \Mat_n(\K)$, we let $\Sp(M)$ denote the set of its eigenvalues \emph{in the field} $\K$.

\begin{theo}[Generalized Gerstenhaber theorem]\label{Gersten}
Let $V$ be a linear subspace of $\Mat_n(\K)$ such that $\Sp(M) \subset \{0\}$ for every $M \in V$.
Then $\dim V \leq \binom{n}{2}\cdot$
\end{theo}

Note that this implies the Gerstenhaber theorem on linear subspaces of nilpotent matrices
\cite{Brualdi,Gerstenhaber,Mathes}, and that this is equivalent to it when $\K$ is algebraically closed.
Moreover, for $\K=\R$, the proof is easy by intersecting $V$ with the space of symmetric matrices of $\Mat_n(\K)$
(see \cite{Meshulam2}).
Our proof for an arbitrary field will use a brand new method. For $i \in \lcro 1,n\rcro$ and $M \in \Mat_n(\K)$, we let $L_i(M)$
denote the $i$-th row of $M$. We set
$$R_i(V):=\bigl\{M \in V : \; \forall j \in \lcro 1,n\rcro \setminus \{i\}, \; L_j(M)=0\bigr\}.$$

\begin{prop}\label{combin}
Let $V$ be a linear subspace of $\Mat_n(\K)$ such that $\Sp(M) \subset \{0\}$ for every $M \in V$.
Then $R_i(V)=\{0\}$ for some $i \in \lcro 1,n\rcro$.
\end{prop}

\begin{proof}
We prove this by induction on $n$. Assume the claim holds for any subspace of $(n-1)\times (n-1)$ matrices satisfying the assumptions,
and that it fails for $V$.
Denote by $W$ the linear subspace of $V$ consisting of its matrices with a zero last row.
We decompose every $M \in W$ as
$M=\begin{bmatrix}
K(M) & ? \\
0 & 0
\end{bmatrix}$. Notice that $K(W)$ is a linear subspace of $\Mat_{n-1}(\K)$ satisfying the assumptions of Proposition \ref{combin}.
By the induction hypothesis, there is an integer $i \in \lcro 1,n-1\rcro$ such that $R_i(K(W))=\{0\}$. However, $R_i(V) \neq \{0\}$,
hence $V$ contains the elementary matrix $E_{i,n}$ (i.e.\ the one with entry $1$ at the spot $(i,n)$, and for which all the other entries are zero).
Conjugating $V$ with a permutation matrix, this generalizes as follows: for every $k \in \lcro 1,n\rcro$, there is an integer $f(k) \in \lcro 1,n\rcro$
such that $E_{f(k),k} \in V$. We may then find an $f$-cycle, i.e.\ a list $(i_1,\dots,i_p)$ of pairwise distinct integers such that
$f(i_1)=i_2$, $f(i_2)=i_3$, \ldots, $f(i_{p-1})=i_p$ and $f(i_p)=i_1$. Hence $V$ contains the matrix
$M:=E_{i_1,i_p}+\sum_{k=1}^{p-1} E_{i_{k+1},i_k}$. However $1 \in \Sp(M)$
(consider the vector with entry $1$ in every $i_k$ row, and zero elsewhere), contradicting our assumptions.
\end{proof}

\begin{proof}[Proof of Theorem \ref{Gersten}]
Again, we use an induction process. The result is trivial when $n=0$ or $n=1$.
Assume $n \geq 2$ and the results holds for subspaces of $\Mat_{n-1}(\K)$.
Let $V \subset \Mat_n(\K)$ be as in Theorem \ref{Gersten}.
Using Proposition \ref{combin}, we lose no generality assuming that $R_n(V)=\{0\}$ (we may
reduce the situation to this one by conjugating $V$ with a permutation matrix).
Consider the linear subspace $W$ of $V$ consisting of its matrices which have the form
$$M=\begin{bmatrix}
A(M) & 0 \\
L(M) & \alpha(M)
\end{bmatrix}$$
where $A(M) \in \Mat_{n-1}(\K)$, $L(M) \in  \Mat_{1,n-1}(\K)$ and $\alpha(M) \in \K$. \\
Then the rank theorem shows that $\dim V \leq (n-1)+\dim W$. \\
For every $M \in W$, one has $\Sp(M) \subset \{0\}$ hence $\alpha(M)=0$ and $\Sp A(M) \subset \{0\}$.
Since $R_n(V)=\{0\}$, this yields $\dim A(W)=\dim W$, whilst the induction hypothesis shows that
$\dim A(W) \leq \binom{n-1}{2}\cdot$
We conclude that
$$\dim V \leq (n-1)+\binom{n-1}{2}=\binom{n}{2}\cdot$$
\end{proof}

\begin{Rem}
Proceeding by induction and using Proposition \ref{combin}, it can even be proven that under the assumptions of Theorem
\ref{Gersten}, there is a permutation matrix $P \in \GL_n(\K)$ such that $(P\,V\,P^{-1})\cap T_n^-(\K)=\{0\}$,
where $T_n^-(\K)$ denotes the space of lower triangular matrices in $\Mat_n(\K)$. This would immediately yield Theorem \ref{Gersten}.
\end{Rem}

\noindent This completes the proof of Theorem \ref{condsuff}.


\begin{thebibliography}{1}

\bibitem{AtkLloyd}
M.D. Atkinson, S. Lloyd, Large spaces of matrices of bounded rank,
\newblock{\em Quart. J. Math. Oxford (2),}
\newblock{\textbf{31}}
\newblock{(1980),}
\newblock{253-262.}

\bibitem{Brualdi}
R. Brualdi, K. Chavey, Linear spaces of Toeplitz and nilpotent matrices,
\newblock{\em J. Combin. Theory Ser A,}
\newblock{\textbf{63}}
\newblock{(1993),}
\newblock{65-78.}

\bibitem{Dieudonne}
J. Dieudonn\'e, Sur une g\'en\'eralisation du groupe orthogonal \`a quatre variables,
\newblock{\em Arch. Math.,}
\newblock{\textbf{1}}
\newblock{(1949),}
\newblock{282-287.}

\bibitem{Flanders}
H. Flanders, On spaces of linear transformations with bounded rank,
\newblock{\em J. Lond. Math. Soc.,}
\newblock{\textbf{37}}
\newblock{(1962),}
\newblock{10-16.}

\bibitem{Gerstenhaber}
M. Gerstenhaber, On Nilalgebras and Linear Varieties of Nilpotent Matrices (I),
\newblock{\em  Amer. J. Math.,}
\newblock{\textbf{80}}
\newblock{(1958),}
\newblock{614-622.}

\bibitem{Mathes}
B. Mathes, M. Omladi\v c, H. Radjavi, Linear spaces of nilpotent matrices,
\newblock{\em  Linear Algebra Appl.,}
\newblock{\textbf{149}}
\newblock{(1991),}
\newblock{215-225.}

\bibitem{Meshulam}
R. Meshulam, On the maximal rank in a subspace of matrices,
\newblock{\em  Q. J. Math., Oxf. II,}
\newblock{\textbf{36}}
\newblock{(1985),}
\newblock{225-229.}

\bibitem{Meshulam2}
R. Meshulam, On two extremal matrix problems,
\newblock{\em  Linear Algebra Appl.,}
\newblock{\textbf{114/115}}
\newblock{(1989),}
\newblock{261-271.}

\bibitem{dSPaffpres}
C. de Seguins Pazzis, The affine preservers of non-singular matrices,
\newblock{\em Arch. Math.,}
\newblock{\textbf{95}}
\newblock{(2010)}
\newblock{333-342.}

\bibitem{dSPclass}
C. de Seguins Pazzis, The classification of large spaces of matrices of bounded rank,
\newblock{\em  ArXiv preprint}
\newblock http://arxiv.org/abs/1004.0298


\end{thebibliography}
\end{document}